\documentclass[12pt,twoside,a4paper]{amsart}
\usepackage{amssymb}
\date{\today}

\def\ddzi{\frac{\partial}{\partial \z _{i}}}
\def\z{\zeta}
\def\Pn{\P^n}

\def\End{{\rm End}}

\def\deg{\text{deg}\,}

\def\w{\wedge}

\def\dbar{\bar\partial}

\def\C{{\mathbb C}}
\def\w{{\wedge}}
\def\P{{\mathbb P}}

\def\D{{\mathcal D}}

\def\S{{\mathcal S}}

\def\Hom{{\rm Hom\, }}
\def\codim{{\rm codim\,}}

\def\Im{{\rm Im\, }}

\def\E{{\mathcal E}}

\def\Ok{{\mathcal O}}

\def\L{{\mathcal L}}
\def\Q{{\mathcal Q}}
\def\Re{{\rm Re\,  }}

\def\L{{\mathcal L}}

\def\H{{\mathcal H}}
\def\J{{\mathcal J}}

\def\be{\begin{equation}}
\def\ee{\end{equation}}
\def\PM{{\mathcal PM}}

\newtheorem{thm}{Theorem}[section]
\newtheorem{lma}[thm]{Lemma}

\newtheorem{prop}[thm]{Proposition}

\theoremstyle{definition}

\newtheorem{df}{Definition}

\theoremstyle{remark}

\newtheorem{preremark}{Remark}
\newtheorem{preex}{Example}

\newenvironment{remark}{\begin{preremark}}{\qed\end{preremark}}
\newenvironment{ex}{\begin{preex}}{\qed\end{preex}}

\numberwithin{equation}{section}

\title[]{Explicit representation of membership in polynomial ideals}

\begin{document}

\date{\today}

\author{Mats Andersson \& Elin G\"otmark}

\address{Department of Mathematics\\Chalmers University of Technology and the University of 
G\"oteborg\\S-412 96 G\"OTEBORG\\SWEDEN}

\email{matsa@math.chalmers.se}

\subjclass{32A26, 32A27, 32B15,  32C30}

\thanks{The first author was
  partially supported by the Swedish 
  Research Council. Support by the Institut Mittag-Leffler (Djursholm,
  Sweden) is also gratefully acknowledged.}

%%{\bf\large Preliminary  version}

\begin{abstract}
We introduce a new division formula on projective space which provides 
explicit solutions to various polynomial division problems with sharp
degree estimates. We consider
simple examples as the classical Macaulay theorem as
well as a quite recent result  by Hickel, related to the effective
Nullstellensatz. We also obtain a related result that generalizes
Max Noether's classical $AF+BG$ theorem. 
\end{abstract}

%%

%%uppsats/integralformler2/genexact.....

%%%%%%%%%%

\maketitle

\section{Introduction}

Let $F_1, \ldots, F_m$ be polynomials in $\C^n$ 
of degrees $d_1\ge d_2\ge\ldots\ge d_m$  and assume that
 $\Phi$ is a polynomial that 
 vanishes on the common zero set of the $F_j$. By Hilbert's Nullstellensatz
one can find polynomials $Q_j$ such that
\begin{equation}\label{busa}
\sum_j F_j Q_j=\Phi^\nu
\end{equation}
if the power $\nu$ is large enough.
A lot of attention has been paid to find effective versions, i.e., control of $\nu$ and
the degrees of $Q_j$ in terms of the degrees of $F_j$.
In  \cite{Br}  Brownawell obtained
bounds on $\nu$ and $\deg Q_j$ not too far from  the best possible, using 
a combination of algebraic and analytic methods.
%% cf., Remark~\ref{brownawell}
%%below.
Soon after that Koll\'ar \cite{Koll} obtained  by purely algebraic 
methods the optimal result: 

{\it If $\deg F_j=d_i$ (and $d_i\neq 2$), then
\eqref{busa} holds for some $\nu\le N_{ko}$ and $\deg(F_iQ_i)\le (1+\deg \Phi)N_{ko}$,
where $N_{ko}$ is $d_1\cdots d_m$ if $m\le n$ and
$N_{ko}=d_1\cdots d_{n-1}\cdot d_m$ otherwise.}

In particular, if $F_j$ have no common zeros in $\C^n$, then there are
polynomials $Q_j$ such that
\begin{equation}\label{bezout}
\sum_j F_jQ_j=1,
\end{equation}
with
$$
\deg   F_jQ_j\le  N_{ko}.
$$
The restriction $d_j\neq 2$ has  been removed by Jelonek, \cite{Jal};
even more interesting is that the method he uses, 
basically elimination theory, 
actually produces explicitly the desired polynomials $Q_i$.
\smallskip

A standard way to reformulate  problems of this kind is the following.
Let $z=(z_0,\ldots,z_n)$, $z'=(z_1,\ldots,z_n)$,
let $f^i(z)=z_0^{d_i}F_i(z'/z_0)$ be the homogenizations of $F_i$, and let
$\phi(z)=z_0^{\deg \Phi}\Phi(z'/z_0)$.
Then there is a solution to 
\begin{equation}\label{uppsala}
\Phi=\sum_i F_i Q_i
\end{equation}
with  $\deg F_i Q_i\le\rho$ if and only if
there are  $(\rho-d_i)$-homogeneous forms $q_i$ such that
\begin{equation}\label{hoppla}
\sum_i f^i q_i=z_0^{\rho-\deg\Phi}\phi.
\end{equation}

\smallskip
In \cite{A3},    \eqref{hoppla}
is considered as an equation in  vector bundles over $\P^n$ and it is 
shown  that if  $\phi$ annihilates a certain residue current
$R$ on $\P^n$, then there is indeed a global solution 
$q=(q_1,\ldots,q_m)$ provided that $\rho$
also satisfies an additional estimate from below.
In this new paper we introduce an explicit  division formula
\begin{equation}\label{div}
\psi(z)=\sum_i f^i(z)\cdot \int_{\P^n} \H_i U\psi+
\int_{\P^n}\H R\psi
\end{equation}
that holds for $\psi=z_0^{\rho-\deg\Phi}\phi$
if $\rho$ is large enough. Here  $U$ is a current
that is smooth outside the common zero set $Z\subset \P^n$ of the $f^i$,
$R$ is a residue current with support on $Z$, 
and $\H_i$ are  smooth in both variables and
homogeneous  polynomials of degrees $\rho-d_i$ in $z$. Thus 
\eqref{div} provides an explicit solution
(after dehomogenization) to \eqref{uppsala}
if $R\psi=0$.  Moreover, in many  cases one can choose the same
$\rho$ as in the implicit method, so we indeed get the same
degree estimate of the solution.

\begin{remark}
 Integral representation
of solutions to polynomial division problems was introduced in \cite{BB} and has been used
since then by many authors,  see, e.g., \cite{BGVY} and the survey article
\cite{TY}. However, in these  formulas the integration 
(or rather the action of a current on a test form) is performed over $\C^n$ so
a size estimate at infinity is needed to get rid of the
residue term. The formula  \eqref{div}  is intrinsic on
 $\P^n$ and so the residue term may vanish for more subtle reasons.
A similar formula  was introduced already
in \cite{A3} but with less precise degree estimates.
\end{remark}

We also notice in this paper that, by a more careful residue calculus,
the method in \cite{A3} admits more general results than stated there.
Therefore  we  start with some
examples of results that can be obtained in this way, and for which we 
now also have explicit representations. 
%%The division  formula is derived  in Section~\ref{formelsec}.
%% and proofs of the theorems
%%below are given in Section~\ref{proofsec}. 
If   $\Phi$ belongs to the integral
closure of  the ideal  $(F_1,\ldots, F_m)$  then it follows from the Brian\c con-Skoda 
theorem, \cite{BS},  that one can take $\nu=\min(m,n)$. The following nice  result was proved
by Hickel, \cite{Hick}, as an affirmative answer to a conjecture  by Berenstein and Yger
in \cite{BY4}.

\begin{thm} \label{hickelthm}
[Hickel] If $\Phi$ is in the integral closure of $(F_1,\ldots, F_m)$, 
then  \eqref{busa} holds with $\nu$ equal to $\min(m,n)$  and 
\begin{equation}\label{hickest}
\deg(F_iQ_i)\le \min(m,n)(\deg \Phi +N_{hi}),
\end{equation}
where $N_{hi}$  is $d_1\cdots d_m$ if $m\le n$ and
$N_{hi}=\min(d_1^n, d_1\cdots d_m/d_m^{m-n})$ otherwise.
\end{thm}

%%To be precise,  in \cite{Hick} there is the factor  $\min(m,n+1)$ instead of $\min(m,n)$
%%that is obtained by our method. 
Notice that, applied to  $\Phi=1$,  one essentially  gets back Koll\'ar's theorem,
except for the factor  $\min(m,n)$ in front of $N_{hi}$. 

%%\smallskip
%%The proof in  \cite{Hick} relies on a global Brian\c con-Skoda
%%theorem, which gives the factor $\min(m,n+1)$ in \eqref{hickest}.
%%In our approach it is enough to annihilate a  residue current
%%and therefore  the local Brian\c con-Skoda constant $\min(m,n)$ is enough. 

\smallskip
The assumption that $\Phi$ is in the integral closure means 
(is equivalent to) that
$|\Phi|\le C|F|$ locally in $\C^n$.
The following slightly stronger statement actually holds:

{\it If $\Phi$ is a polynomial
such that
\begin{equation}\label{puss}
|\Phi|\le C |F|^{\min(m,n)}
\end{equation}
locally in $\C^n$,  then  \eqref{uppsala}  has a solution such that
$$
\deg(F_iQ_i)\le \deg\Phi + \min(m,n)N_{hi}.
$$}
%%
%% the other hand the theorem is a corollary of a much sharper result.

Notice that the condition \eqref{puss}, thanks to the Brian\c con-Skoda theorem, 
implies that
$\Phi$ is in the ideal $(F)$ in $\C^n$.

\smallskip
To prove Theorem~\ref{hickelthm}  Hickel introduces a nonnegative  rational number
$\nu_\infty$ that is a measure of the ``order of contact''  of 
$Z$ to the hyperplane at infinity,
see Section~\ref{proofsec} for the precise definition.
%%%
From the Refined Bezout theorem due to Fulton and MacPherson, see \cite{Fu}, Hickel
deduces, by quite  rough estimates, that 
\begin{equation}\label{rough}
\nu_{\infty} \le N_{hi};
\end{equation}
in most cases $\nu_{\infty}$ is much smaller. 
Theorem~\ref{hickelthm} is then an immediate corollary of  the following
result.
As usual,  $\ulcorner\alpha\urcorner$ denotes  the least integer
that is $\ge \alpha$.

\begin{thm} \label{hickelthm2}
 Assume that  $\Phi$ satisfies \eqref{puss}.

\smallskip
\noindent (i)  Then \eqref{uppsala} has a solution such that
\begin{equation}\label{hickest2}
\deg(F_iQ_i)\le \deg \Phi +\ulcorner\min(m,n)\nu_{\infty}\urcorner
\end{equation}
if $m\le n$. In case $m>n$ the estimate is the maximum of the right hand side
and the number $d_1+\ldots+ d_{n+1}-n$.

\smallskip
\noindent (ii) The integral formula provides  a   solution to  \eqref{uppsala}
with 
\begin{multline}\label{hickest3}
\deg(F_iQ_i)\le 
\\
\max\big(\deg \Phi +\ulcorner \min(m,n)\nu_{\infty}\urcorner, d_1+\cdots +d_{min(m,n+1)}-n\big).
\end{multline}
%%%
\end{thm}

Part $(i)$ is a variant of Theorem~2.1' in \cite{Hick}.
By an extra argument due to Hickel one can get rid  of the
$\ulcorner\ \urcorner$ in \eqref{hickest2} and thus gain
one unit when the number inside is not an integer.

\smallskip
We can  assume from the beginning 
that $\Phi$ is in the ideal $(F_1, \ldots, F_m)$  (but not assuming
\eqref{puss}) and ask for  an estimate
of the degrees of $Q_j$. Here is (as far as we know) a new result in this direction:

\begin{thm}\label{newthm}
Assume that $\codim\{F_1=\ldots= F_m=0\}\ge m $ in $\C^n$ and that 
$\Phi$ is in the ideal $(F_1,\ldots,  F_m)$ in $\C^n$.  
 Then \eqref{uppsala} has a solution such that
\begin{equation*} %%\label{hickest2}
\deg(F_iQ_i)\le \deg \Phi +\ulcorner m\nu_{\infty}\urcorner,
\end{equation*}
whereas the integral formula provides  a   solution 
with 
$$
\deg(F_iQ_i)\le 
\max\big(\deg \Phi +\ulcorner m\nu_{\infty}\urcorner, d_1+\cdots +d_m-n\big).
$$
\end{thm}

%%If  the polynomial $\Psi$ vanishes on the zero set 
%%of  $(F_1, \ldots, F_m)$ in $\C^n$, $m\le n$, then $\Phi=\Psi^{d_1\cdots d_m}$
%%is in the ideal,  cf., Corollary~1.4 in \cite{A3}.
%%
%%SO???????????  WHAT !!!!!

\begin{remark}
If $Z$ has no  irreducible component at all contained in the
hyperplane at infinity, then $\nu_\infty$  shall be
interpreted as  $0$ in Theorem~\ref{newthm},   and then  we get back
Theorem~1.2 in \cite{A3}. In case   $m=n$ this is
the classical so-called $AF+BG$ theorem of Max Noether, \cite{Noe}.

\smallskip
In particular, if  $m\ge n+1$ and $Z$ is empty,  then
we get back the classical Macaulay  theorem, \cite{Macaul}, with an explicit
formula representing the membership. Such a formula has also been obtained
in \cite{Boy} relying on \cite{BH}.
\end{remark}

We have
analogous results for submodules of $\C[z_1,\ldots,z_n]^r$ 
 rather than just ideals.
Let $F$ be a polynomial mapping $\C^n\to\Hom(\C^m,\C^r)$ whose
columns  $F^j$ have degrees $\le d_j$, $j=1,\ldots,m$. We also assume
that $F$  is generically surjective in $\C^n$, which is equivalent to that
the ideal $\det  F$, generated by the $r\times r$ minors of $F$,
is nontrivial.
 Let $f$ be the
associated matrix whose columns $f^j$ are $d_j$-homogeneous polynomials 
and let $Z$ be the set in $\P^n$ where $\det f$ is vanishing.
Moreover, let  $\nu_{\infty}$  be associated to the ideal
$\det f$. By the estimate \eqref{rough} above we get (without being too precise)
$$
\nu_{\infty}\le   (d_1\cdots d_r)^{\min(n, m!/(m-r)!r!)}.
$$
Let $\Phi$ be an $r$-column of polynomials.
The proper analogue of \eqref{puss} is, cf.,  \cite{A4}, 
that 
\begin{equation}\label{matrisint}
\|\Phi\|\le C|\det F|^{\min(n,m-r+1)}
\end{equation}
holds locally, 
where $\|\Phi\|$ is a somewhat stronger norm than the natural norm $|\Phi|$;
i.e.,  $\|\Phi\|\le |\Phi|$.
We have the following generalization of the previous
theorems.  %%Theorem~\ref{    },

\begin{thm} \label{hickelthm2matris}
 Assume that  $F$ is an $r\times m$ matrix of polynomials as above
with columns $F^j$
and $\Phi$ is an $r$-column of polynomials.
Assume that 
either \eqref{matrisint} holds  locally in $\C^n$, 
or that 
\begin{equation}\label{bolly}
\codim\{\det  F=0\}\ge m-r+1  
\end{equation}
in $\C^n$ and  
$\Phi$ is in the module generated by the columns $F^j$.

\smallskip
\noindent (i)  There are polynomials $Q_j$ such that 
\eqref{uppsala} holds and 
\begin{equation}\label{hickest5}
\deg(F^iQ_i)\le \deg\Phi + \ulcorner\min(n,m-r+1)\nu_{\infty}\urcorner
\end{equation}
if $m\le n+r-1$. 
In case $m>n$ the estimate is the maximum of the right hand side
and the number $d_1+\cdots + d_{n+r} -n$.

\smallskip
\noindent (ii) The division formula provides a  solution to  \eqref{uppsala}
such that 
\begin{multline}\label{hickest3matris}
\deg(F^iQ_i)\le 
\\
\max\big(\deg \Phi +\ulcorner\min(n,m-r+1)\nu_{\infty}\urcorner, d_1+\cdots +
d_{\min(m,n+r)}-n\big).
\end{multline}
%%%
\end{thm}

Notice that for a generic $r\times m$-matrix $F$, the zero set
of $\det F$ has codimension $m-r+1$. 

\smallskip
Again $\nu_\infty$ is $0$ if \eqref{bolly} holds and
$Z$ has no irreducible component
contained in the hyperplane at infinity
\smallskip

One can  obtain sharper  results for special ideals, 
for instance  determinantal ideal, and product ideals,  cf., \cite{A8}.
However we skip precise formulations since our aim
is mainly  to give examples of  
various applications of the new representation formulas.

\smallskip
We are grateful to Michel Hickel and Alain Yger for important remarks
on a preliminary version of this paper. We are also grateful to the
anonymous referee for his careful reading and valuable suggestions.

\section{The  basic setup}\label{basic}

Assume that 
\begin{equation}\label{ecomplex}
0\to E_N\stackrel{f_N}{\longrightarrow}\ldots\stackrel{f_3}{\longrightarrow} 
E_2\stackrel{f_2}{\longrightarrow}
E_1\stackrel{f_1}{\longrightarrow}E_0\to 0
\end{equation}
is a generically exact complex of Hermitian vector bundles  over $\P^n$ and
let $Z$ be the algebraic set where \eqref{ecomplex}  is not pointwise exact. 
In \cite{AW1} were  introduced   currents 
$$
U=U_1+\ldots +U_N, \quad R=R_1+\ldots+ R_N
$$
associated to \eqref{ecomplex} with the following properties:
The current $U$ is smooth outside $Z$, $U_k$ are $(0,k-1)$-currents
that take values in $\Hom(E_0, E_k)$,  and
$R_k$ are  $(0,k)$-currents with support on $Z$, taking values in 
$\Hom(E_0, E_k)$. Moreover, they satisfy the relations
\begin{equation}\label{alban}
f_1 U_1=I_{E_0}, \quad  f_{k+1}U_{k+1}-\dbar U_k= -R_k,\ k\ge 1,
\end{equation}
which can be compactly written as
$
\nabla_f U=I_{E_0}-R
$
if 
$
\nabla_f=f-\dbar=f_1+f_2+\cdots f_N-\dbar.
$
We  have the  corresponding complex of locally free sheaves 
$$
0\to \mathcal{O}(E_N)\stackrel{f_N}{\longrightarrow}\ldots\stackrel{f_3}{\longrightarrow} 
\mathcal{O}(E_2)\stackrel{f_2}{\longrightarrow}
\mathcal{O}(E_1)\stackrel{f_1}{\longrightarrow}\mathcal{O}(E_0).
$$
If $\psi$ is  a holomorphic section of $\mathcal{O}(E_0)$ 
that annihilates $R$, i.e., the current $R\phi$ vanishes,  then $\psi$ is 
in the sheaf $\J=\Im f_1$, see \cite{AW1}, Proposition~2.3.

In this paper we will only consider bundles that are direct sums of line bundles.
Let $\mathcal{O}(\ell)$ be the holomorphic line bundle over $\P^n$ whose sections are 
(naturally identified with)  $\ell$-homogeneous functions in $\C^{n+1}$.
Moreover, let $E_{k}^j$ be disjoint trivial line bundles over $\P^n$ 
with basis elements $e_{k,j}$, and let
\begin{equation}\label{dirsum}
E_k=\big(E^1_k\otimes \mathcal{O}(-d^1_k)\big)\oplus\cdots \oplus 
\big(E^{r_k}_k\otimes \mathcal{O}(-d_k^{r_k})\big).
\end{equation}
Then 
$$
f_k=\sum_{ij} f_k^{ij}e_{k-1,i}\otimes e_{k,j}^*
$$
are matrices of homogeneous forms, here $e_{k,j}^*$ are the dual elements,  and 
$$
\deg f_k^{ij}= d_{k-1}^i-d_k^j.
$$
%%%
We equip $E_k$ with the natural  Hermitian metric, i.e., such that 
$$
|\xi(z)|^2_{E_k}=\sum_{j=1}^{r_k}|\xi_j(z)|^2 |z|^{2d^j_k},
$$
if  $\xi=(\xi_1,\ldots,\xi_{r_k})$.

The Dolbeault cohomology $H^{0,q}(\P^n, \mathcal{O}(r))$ vanishes for all
$r$ if $1\le q\le n-1$ and also for $q=n$ if $r\ge -n$.
From \eqref{alban} and a simple homological argument
we get the following proposition,
see \cite{AW1} for details.

\begin{prop}\label{ormvrak}
Assume that  $\psi$ is a holomorphic section of $\mathcal{O}(\rho)\otimes E_0$ such that 
$R\psi=0$.
If 
\begin{equation}\label{alfred}
N\le n \  \  \text{or}\  \  \rho -\max_i d_{n+1}^i \ge -n,
\end{equation}
then there is a global holomorphic section $q$
of  $\mathcal{O}(\rho)\otimes E_1$ such that  $f_1 q=\psi$.
\end{prop}

Given an ideal sheaf $\J=(f^1,\ldots, f^m)$ as above
one can find a complex \eqref{ecomplex} such that
$\psi$ annihilates the associated residue current $R$
if and only if $\psi$ is in $\J$, see \cite{AW1} Section~7. 
However, in general the 
numbers $d_{n+1}^i$ may be very big and \eqref{alfred} 
then reflects the   obstruction for global solvability.  
For the purpose of this paper we  need 
more specialized complexes.
The Koszul complex will be of particular importance.

\begin{ex}[The Koszul complex]\label{koszul}
Let $f^1,\ldots, f^m$ be our given  homogeneous forms of  degrees $d_i$,
assume that $E_0$ is the trivial line bundle,  and let 
$$
E=(E^1\otimes\mathcal{O}(-d_1))\oplus
\cdots\oplus (E^{m}\otimes\mathcal{O}(-d_m)),
$$
where $E^i$ are trivial line bundles.
Let $e_i$   be basis elements  for $E^i$ and let $e_i^*$ be the dual basis elements.
We  take  
$$
E_k=\Lambda^k E=\sum'_{|I|=k}\mathcal{O}(-(d_{I_1}+\cdots + d_{I_k}))E^{I_1}\otimes\cdots
\otimes E^{I_k}
$$
and $f_k$ as interior multiplication with $f=\sum f^j e_j^*$. 
Now
$$
\sigma=\sum_j\frac{\overline{f^j(z)}}{|z|^{2d_j}}e_j/|f|^2_{E^*}
$$
is the section of $E$ with minimal norm such that $f\cdot \sigma=1$ outside $Z$, 
where
$$
|f(z)|_{E^*}^2=\sum_1^m|f^j(z)|^2|z|^{-2d_j}.
$$
Moreover, 
$$
U=|f|^{2\lambda}_{E^*}\sum_{k=1}^{m} \sigma\w (\dbar \sigma)^{k-1} \Big|_{\lambda=0}
$$
and
\begin{equation}\label{Rdef}
R=
\dbar |f|^{2\lambda}_{E^*}\w \sum_{k=1}^m \sigma\w (\dbar \sigma)^{k-1}\Big|_{\lambda=0},
\end{equation}
cf., \cite{A2} and \cite{A3}; here $|_{\lambda}$ means evaluation at $\lambda=0$ after
analytic continuation (the existence of which  is nontrivial
and part of the statement).
%% meaning that the expressions a~priori
%%defined only for $\Re\lambda>>0$ admit analytic continuations
%%so that thhey can be evaluated unambiguously at $\lambda=0$.
%% where means analytic continuation evaluated at
%%$\lambda=0$. 

\smallskip

The condition \eqref{alfred} in this case is
\begin{equation}\label{alfred2}
m\le n\    \   \text{or}\ \   \rho -(d_1+\ldots +d_{n+1})\ge -n.
\end{equation}

%%The resulting residue current $R$ is of Bochner-Martinelli type

If $\codim Z=m$, then the resulting residue current $R$
only  consists of the term $R_m$ (cf., Proposition~\ref{uppstuds} below); it 
coincides with the classical Coleff-Herrera product,
 cf., \cite{A3} p.\ 112, and hence the 
the duality theorem, \cite{P1} and \cite{DS}, 
holds, i.e., locally
$R\psi=0$ if and only if $\psi$ (locally) belongs to the ideal
sheaf $\J$ generated by $f$.
\end{ex}

\begin{ex}\label{ex2} 
These formulas become much simpler if we assume that all $d_j=d$. Then 
$$
\sigma=\sum_{j=1}^m\frac{\overline{f^j(z)} e_j}{|f(z)|^2}.
$$
If  $\bar f\cdot e=\sum \bar f^j e_j$ and $d\bar f\cdot e=\sum d\bar f^j\w e_j$, then
$$
\sigma\w(\dbar\sigma)^{k-1}=\frac{\bar f\cdot e\w (d\bar f\cdot e)^{k-1}}{|f|^{2k}},
$$
and it is easy to check that we can replace $|f|^{2\lambda}_{E^*}$ by $|f|^{2\lambda}$
in the definition of $U$ and $R$. Thus
$$
U=|f|^{2\lambda}\sum_{k=1}^{m} \frac{\bar f\cdot e\w (d\bar f\cdot e)^{k-1}}{|f|^{2k}} \Big|_{\lambda=0}
$$
and
\begin{equation}\label{Rdef2}
R=
\dbar |f|^{2\lambda}\w \sum_{k=1}^m \frac{\bar f\cdot e\w (d\bar f\cdot e)^{k-1}}{|f|^{2k}}
   \Big|_{\lambda=0}.
\end{equation}
%%Notice that these currents, via the natural projection, are identical 
%%to  the corresponding currents
%%defined directly on $\C^{n+1}\setminus\{0\}$ from the homogeneous forms $f^j$ on $\C^{n+1}$.
\end{ex}

In \cite{AW2} was introduced the sheaf of {\it pseudomeromorphic} currents $\PM$
on a complex manifold $X$. It is a module over the sheaf of smooth forms and
closed under $\dbar$. For any $T\in\PM$ and variety $V$ there exists a
``restriction'' ${\bf 1}_VT$ that is in $\PM$ and has support on $V$.
The current ${\bf 1}_{X\setminus V}T=T-{\bf 1}_VT$ is determined
by the natural restriction of $T$ to  the open set $X\setminus V$ or equivalently,
$T={\bf 1}_VT$ if and only if $T$ has support on $V$. If $h$ is any tuple
of holomorphic functions with common zero set $V$, then
\begin{equation}\label{restdef}
{\bf 1}_{X\setminus V}T=|h|^{2\lambda} T|_{\lambda=0},
\end{equation}
where, as before, the right hand side is the evaluation at the origin
of a  current-valued holomorphic function.
We also have the following important fact (Corollary~2.4 in \cite{AW2}):

\begin{prop}\label{uppstuds}
If $T\in\PM$
has bidegree $(*,p)$ and support on a variety $T$ of codimension $k>p$,
then $T=0$.
\end{prop}

For instance, the currents  $R$ and $U$ above are  pseudomeromorphic.

\section{Proofs of Theorems~\ref{hickelthm2} and \ref{newthm}}\label{proofsec}

Let $\J$ denote the sheaf over  $\P^n$ generated by our  homogeneous forms
$f^1,\ldots, f^m$, let 
\begin{equation}\label{blowup}
\nu\colon X_+\to \P^n
\end{equation}
be the normalization of the blowup along $\J$, and let 
$$
Y^+=\sum_i m_i^+ Y_i^+
$$
be the associated divisor in $X_+$, where $Y^+_i$ are its irreducible
components.   The varieties $\nu(Y^+_i)$ in
$\P^n$ are the so-called distinguished varieties associated to $\J$. 
For each $Y_i^+$ such that $\nu (Y_i^+)$ is entirely contained
in  the hyperplane at infinity,  $H=\{z_0=0\}$,  we let
$e_i$ be the order of vanishing of $\nu^*h$ at $Y_i^+$, where
$h(z)=z_0$, and  define
$$
\nu_{\infty}=\max \frac{m_i^+}{e_i},
$$
where the maximum is  taken over all such indices $i$. 

\smallskip

\smallskip
Let us now  consider the residue current $R$ in  \eqref{Rdef} obtained from the
Koszul complex.
Assume  that $\pi\colon\tilde X\to  \P^n$ is a log resolution of $\J$,
i.e., $\tilde X$ is a smooth modification
such that  $\pi^*\J$ is principal and its zero set $Y$ has 
normal crossings.  Then 
$$
\pi^*f= f^0  f',
$$ 
where 
$ f^0$ is a holomorphic section of a line bundle $ L\to  \tilde X$ that
defines the  divisor 
$$
Y=\sum_i m_i Y_i
$$
and $f'$ is a non-vanishing section of $\pi^*E^*\otimes  L^{-1}$. Moreover, 
$$
L=\otimes_i L_i^{m_i},
$$
and
$$
f^0=\otimes_i s_i^{ m_i},
$$
where $s_i$ are sections of $L_i$ that vanish  to the first order;
the normal crossing assumption means that
in local trivializations the $s_i$ are part of a local holomorphic coordinate system.
It follows that 
$$
\pi^*\sigma=(1/ f^0)\sigma',
$$
where $1/f^0$ is a meromorphic section of $L^{-1}$ and
$\sigma'$ is a smooth section of $\pi^*E\otimes L$.
%%%
Now (in this section $|\cdot |$ always refers to vector bundle norm)
\begin{equation}\label{hopplask}
\dbar |\pi^*f|^{2\lambda}\w \pi^* u_k=
\dbar |\pi^*f|^{2\lambda}\w \frac{1}{(f^0)^k}  u'_k,
\end{equation}
where
$$
u'_k=\sigma'\w(\dbar\sigma')^{k-1}
$$
is smooth. It follows that  \eqref{hopplask} has an analytic continuation to
$\lambda=0$  and the value there is
$$
\tilde R=\dbar \frac{1}{(f^0)^k}\w u_k';
$$
moreover, 
cf., \eqref{Rdef}, 
$$
R_k= \pi_*\tilde R_k.
$$
Since we have normal crossings we also have the decomposition
\begin{equation}\label{bella1}
\tilde R_k=\sum_j\tilde  R_{kj}=
\sum_j \dbar\frac{1}{s_j^{k  m_j}}\w \otimes_{i\neq j}
\frac{1}{s_i^{k m_i}}   u'_k
\end{equation}
and hence
\begin{equation}\label{bella2}
R_k=\sum_j R_{kj}=\sum_j \pi_*\tilde R_{kj}.
\end{equation}

\noindent {\bf Claim:}  $R_{kj}$ vanishes unless $\pi(Y_j)$ is 
a  distinguished variety of $\J$.
%%%
\smallskip

To see this first notice that in  $X_+$, 
$$
\nu^*f=f_+^0 f_+',
$$
where $f_+^0$ is a holomorphic section of the line bundle $L_+\to  X_+$ that
defines the divisor $Y^+$ and $f_+'$ is a non-vanishing section of
$\nu^*E^*\otimes L^{-1}$.
The log resolution  $\pi$   factorizes over $\nu$, i.e., we have  
\begin{equation}
\tilde X\stackrel{\tilde\nu}{\to} X_+\stackrel{\nu}{\to} \P^n.
\end{equation}
If  $\pi(Y_j)$
is not a distinguished variety of $\J$, then 
$\tilde\nu(Y_j)$ has at least codimension $2$ in $X_+$. 
Therefore
$$
\tilde\nu_* \big[\dbar\frac{1}{s_j^{k  m_j}}\w \otimes_{i\neq j}
\frac{1}{s_i^{k m_i}}\big]
$$
must vanish according to Proposition~\ref{uppstuds}
since it is a pseudomeromorphic $(0,1)$-current in $X_+$ 
with support on a variety with codimension at least $2$. (In  general
$X_+$ is not smooth but the proof of Proposition~\ref{uppstuds}
goes through verbatim even in the non-smooth case.) 
Notice that 
$$
\nu^*u_k=\frac{1}{(f_+^0)^k}u_{+,k}',
$$
where $u_{+,k}'$ is smooth,  and that 
$\tilde u_k'=\nu^* u_{+,k}'$.
It follows that
$
\tilde\nu_* \tilde R_{kj}=0
$
and hence
$
\pi_* \tilde R_{kj}=\nu_*\tilde\nu_* \tilde R_{kj}=0
$
as claimed.  

\smallskip
The resulting   decomposition of $R_k$ with respect to the distinguished varieties
is inspired by \cite{W} and \cite{JW}, where it  plays a fundamental role. 

\smallskip
We are now ready for the proofs.
We have already observed that
Theorem~\ref{hickelthm} follows from Theorem~\ref{hickelthm2}.

\begin{proof}[Proof of Theorem~\ref{hickelthm2}]
Let $\phi$ be the homogenization of $\Phi$ and let $\mu=\min(m,n)$.
By the definition of $\nu_\infty$,
we have that $\nu^* h^{\ulcorner\mu\nu_\infty \urcorner}$ must vanish to
order $\mu m_i^+$ on each  divisor $Y_i^+$ such that $\nu(Y^+_i)\subset H$. 
On the other hand, it follows  from \eqref{puss} that  
$\nu^* \phi$ vanishes to order $\mu m_i^+$ on any  other divisor $Y_i^+$.
Thus $\nu^*(\phi h^{\ulcorner\mu\nu_\infty \urcorner})$ vanishes
to order $\mu m_i^+$ on each  $Y_i^+$ and  therefore
\begin{equation}\label{olle}
|\phi h^{\ulcorner\mu\nu_\infty \urcorner}|\le C|f|^\mu
\end{equation}
on $\P^n$.  
Thus $\pi^*(\phi h^{\ulcorner\mu\nu_\infty \urcorner})$ must contain the factor
$s_j^{\mu m_j}$, which implies that it annihilates $\tilde R_{kj}$,
cf., \eqref{bella1} and \eqref{bella2},
for each $k\le\mu$. It follows that $\phi h^{\ulcorner\mu\nu_\infty \urcorner}$
annihilates the current $R$.
 
Now  $\phi h^{\ulcorner\mu\nu_\infty \urcorner}$ is a section of $\mathcal{O}(\rho)$
with $\rho=\deg\Phi+\ulcorner\mu\nu_\infty \urcorner$.
If necessary we raise the power of $h$ further so that \eqref{alfred2} holds.
Then  the first part of Theorem~\ref{hickelthm2} follows from 
Proposition~\ref{ormvrak} after dehomogenization. The  
second part concerning explicit representation follows
from Section~\ref{formelsec}. 
\end{proof}

\begin{remark}
Let $\bar\J$ denote the integral closure sheaf generated by $\J$. 
It is well known that $\xi\in \mathcal{O}_{\P^n}$ is in $\bar\J$ if and only if
$\nu^*\xi$ vanishes to order (at least) $m_i^+$ on $Y_i^+$, i.e., 
$$
\bar\J=\nu_*(\mathcal{O}(-Y^+)).
$$
In the same way, \eqref{olle} means that 
$\phi h^{\ulcorner\mu\nu_\infty \urcorner}$ belongs to the
integral closure of $\J^\mu$.
\end{remark}

\begin{remark}
One can get  rid of $\ulcorner\ \urcorner$ in \eqref{hickest2}  by the  following trick.
Assume that $\mu\nu_\infty =a/b$,  for integers $a,b$.
Let $\hat f(z)=f(z_0^b,\ldots,z_n^b)$ and similarly with $\hat\phi$.
Now 
$$
|h^a\hat\phi|\le C|\hat f|^\mu
$$
and as before we then have a solution $\Psi_j$ to 
$$
\sum_j\hat F^j \Psi_j=\hat\Phi
$$
with
$$
\deg \hat F^j\Psi_j\le \deg\hat\Phi+a=b\deg\Phi+a.
$$
However, one can then choose $\Psi$ of the form
$\Psi=\hat Q$, and it follows that
$$
\deg F^jQ_j\le \deg\Phi+a/b
$$
as desired. See \cite{Hick} for details.
\end{remark}

In the previous proof we killed  the residue by size estimates  in $\C^n$ 
as well as on the hyperplane $H$. 
In the proof of  Theorem~\ref{newthm} the residue calculus is  somewhat more involved
because then we will kill differen parts of the residue in different ways.

\begin{proof}[Proof of Theorem~\ref{newthm}]
We begin with the same set-up as in the previous proof.
Since $R$ is pseudomeromorphic, cf., Section~\ref{basic}
above,  we have the decomosition
$$
R={\bf 1}_{\C^n}R+{\bf 1}_HR,
$$
where ${\bf 1}_{\C^n}R$ is the natural extension to $\P^n$ of the
restriction of $R$ to $\C^n$. Since $f$ is a complete intersection here,
$R$  coincides with the Coleff-Herrera product, and by the duality theorem,
it follows that $h^{\rho-\deg\Phi}\phi{\bf 1}_{\C^n} R=0$.

Assume now that  $Z$ has no irreducible component contained in $H$.
Then $Z$ has codimension $m$ in $\P^n$ and by Proposition~\ref{uppstuds}
hence $R=R_m$.
Thus ${\bf 1}_HR={\bf 1}_HR_m$ has bidegree $(0,m)$ and support
on $H\cap Z$ that has codimension strictly larger than $m$. By
Proposition~\ref{uppstuds} it must therefore vanish. 
It follows that $h^{\rho-\deg\Phi}\phi R =0$ and since \eqref{alfred2}
is satisfied, the membership  follows.

\smallskip
We now consider the general case. We can choose the log resolution so that
also  $\pi^*h$ is a monomial in $s_i$.  %%%%$\tilde X$. 
Notice, cf., \eqref{bella1} and \eqref{bella2}, that 
$$
|h|^{2\lambda}R_{jk}= \pi_* \big(|\pi^*h|^{2\lambda}\tilde R_{jk}\big)
$$
vanishes when  $\Re \lambda$ is large and hence for $\lambda=0$
if $s_j$ is a factor in $\pi^* h$, whereas the value at $\lambda=0$
is $R_{jk}$ if $s_j$ is not a factor in $\pi^* h$.
In view of \eqref{restdef} 
it follows that  ${\bf 1}_HR$ is the sum of 
$R_{kj}$ such that $\pi(Y_j)$ is
contained in $H$. 
Moreover, we  know that 
only $j$ corresponding to distinguished varieties give a contribution. 
Take such a $j$ and assume that $\tilde\nu(Y_j)=Y_i^+$.
We know that $\nu^*h^{\ulcorner\mu \nu_{\infty}\urcorner}$ vanishes  
at least to the same order as $\nu^*f^\mu$ does on $Y_i^+$, and hence
$\pi^*h^{\ulcorner\mu \nu_{\infty}\urcorner}$ must vanish at least to the 
same order as $\pi^*f^\mu$ on $Y_j$, i.e., 
$\pi^*h^{\ulcorner\mu \nu_{\infty}\urcorner}$ contains the factor
$s_j^{\mu m_j}$.
It follows that $\tilde R_{kj} \pi^*h^{\ulcorner\mu \nu_{\infty}\urcorner}=0$.

Summing up, we have that 
$Rh^{\ulcorner\mu \nu_{\infty}\urcorner}\phi=0$, and hence the
first part of the theorem is proved. The second part, again, is
treated in Section~\ref{formelsec}.
\end{proof}

The proof of Theorem~\ref{hickelthm2matris} is analogous, but
one has to use the so-called Buchsbaum-Rim complex which is
a generalization of the Koszul complex. See
\cite{A4} where the associated  currents  are discussed, and
special cases of Theorem~\ref{hickelthm2matris} are  proved.

\section{Explicit representation}\label{formelsec}

We first discuss, following  \cite{A1}, \cite{A7}, \cite{EG0},  and \cite{EG}, 
how one can generate representation
formulas for   holomorphic sections of a  vector bundle  $F\to\P^n$. 
Let $F_z$ denote the pull-back of $F$  to $\Pn_\zeta\times \Pn_z$ under the
natural projection $\Pn_\zeta\times \Pn_z\to \Pn_z$  
and define  $F_\zeta$  analogously.
Notice that 
\[
\eta = 2 \pi i \sum_0^n z_i \ddzi
\]
is a section of the bundle $\mathcal{O}_z(1) \otimes \mathcal{O}_\zeta (-1)
\otimes T_{1,0}(\Pn_\zeta)$ over $\P^n_\zeta\times\P^n_z$. 
If we express a projective form in homogeneous coordinates and contract 
with $\eta$  we get a new projective form, i.e., we have  a mapping 
\[
\delta_\eta : \mathcal{D}'_{\ell+1,q}(\mathcal{O}_\z(k) \otimes \mathcal{O}_z(j)) \to
  \mathcal{D}_{\ell,q}'(\mathcal{O}_\z(k-1) \otimes \mathcal{O}_z(j+1)),
\]
where $\mathcal{D}'_{\ell,q}(\mathcal{O}_\z(k) \otimes \mathcal{O}_z(j))$ denotes the
sheaf of currents on $\P^n_\zeta\times\P^n_z$
of bidegree $(\ell,q)$ in $\zeta$ and $(0,0)$ in $z$
that take  values in $\mathcal{O}_\z(k) 
\otimes \mathcal{O}_z(j)$. 
Given a vector bundle $L\to\P^n_\zeta\times\P^n_z$, let
\[
\L^{\nu}(L)=
\bigoplus_{j}\D'_{j,j+\nu}(\mathcal{O}_\z(j) \otimes \mathcal{O}_z(-j)\otimes  L).
\]
If $\nabla_\eta = \delta_\eta - \dbar$, 
where $\dbar=\dbar_\zeta$, 
then  $\nabla_\eta: \L^{\nu}(L) \to \L^{\nu+1}(L)$. Furthermore,
$\nabla_\eta$ obeys Leibniz' rule, and $\nabla_\eta^2 = 0$.

\smallskip
%%Following \cite{A??} and 
A  {\it weight} with respect to $F\to \P^n$  and a point $z\in\P^n$  
is a section  $g$ of $\L^{0}(\Hom(F_\zeta,F_z))$ such that $\nabla_\eta g =
  0$, $g$ is smooth for $\zeta$ close to $z$, 
and  $g_{0,0} = I_F$ when $\zeta=z$, %%on the diagonal in $\P^n_\zeta\times\P^n_z$,
where $g_{0,0}$ denotes the component of  $g$ with bidegree $(0,0)$,
and $I_F$ is the identity endomorphism on $F$.

\begin{prop} \label{hatsuyuki}
Let $g$ be a weight with respect to $F\to\P^n$ and  $z$,  and assume that
$\psi$ is  a holomorphic section of $F\otimes\mathcal{O}(-n)$.
We  then have  the representation formula
\[
\psi(z) = \int_{\Pn} g_{n,n} \psi. 
\]
\end{prop}

This formula appeared in \cite{EG0}
(Proposition~5.5), and it can be deduced  from \cite{EG};  
however, for the reader's convenience we sketch  a direct  argument.

\begin{proof}
Notice that 
$$
b=\frac{1}{2\pi i}\frac{|\zeta|^2\bar z\cdot d\zeta -(\bar z\cdot\zeta)\bar\zeta\cdot d\zeta}
{|\zeta|^2|z|^2-|\bar\zeta\cdot z|^2}
$$
is a $(1,0)$-form with values in  $\Ok_z(-1)\otimes\Ok_\zeta(1)$ that is smooth 
outside the diagonal
$\Delta$ in $\P^n_\zeta\times\P^n_z$.
Fix the point $z$ and let
$B=b+ b\w\dbar b+\cdots +b\w(\dbar b)^{n-1}$.
Then $\nabla_\eta B=1$ outside $z$, and we claim  that 
\begin{equation}\label{pyrt}
\nabla_\eta B=1-[z]
\end{equation}
 in the current sense, where $[z]$ is the  current
such that 
$$
\int_{\P^n} [z] \xi(\zeta)=\xi(z)
$$
for each smooth section $\xi$ of $\Ok(-n)$. Because of rotational invariance
it is enough to choose affine coordinates
$\zeta=(1,\zeta')=(1,\zeta_1',\ldots,\zeta_n')$ 
and assume  that $z=(1,0,\ldots,0)$.
Then $b$ becomes 
$$
b'=\frac{1}{2\pi i}\frac{\overline{\zeta'}\cdot d\zeta'}{|\zeta'|^2},
$$
and it is elementary to check, cf., \cite{A1} p.\ 5, that 
$(\delta_{\zeta'}-\dbar)(b'+b'\w \dbar b'+\cdots +b'\w(\dbar b')^{n-1}=1-[0]$.
Now  \eqref{pyrt} follows.
Thus
$$
\nabla_\eta(B\w g)=g-[z]\w g=g-[z] g_{0,0}=g-[z]I_F,
$$
and identifying  the top degree terms we get 
$$
\dbar(B\w g)_{n,n-1}=[z]I_F-g_{n,n}.
$$
Multiplying with $\psi$,  the proposition follows from Stokes' formula.
\end{proof}

Consider now a complex \eqref{ecomplex}. 
In order to represent membership in $\J=\Im f_1$,
we will find  a weight $g$ that contains $f_1(z)$ as a factor
and apply Proposition \ref{hatsuyuki}. 
To this end we introduce  a generalization of   so-called Hefer forms,
inspired by \cite{A7} and \cite{AW1}, 
to the case of non-trivial vector bundles.

\df{We say that $H= (H_k^\ell)$ is a Hefer morphism for the complex
  $E_\bullet$  in \eqref{ecomplex}  if $H_k^\ell$ are smooth sections of
\[
\L^{-k+\ell}(\Hom(E_{\zeta,k},E_{z,\ell}))
\]
that are holomorphic in $z$, 
$H_k^\ell=0$ for $k<\ell$, the term $(H_\ell^\ell)_{0,0}$ of bidegree $(0,0)$
is the identity $I_{E_\ell}$ on the diagonal $\Delta$, and
\begin{equation}\label{hrel}
\nabla_\eta H_k^\ell =H_{k-1}^\ell f_k -f_{\ell+1}(z) H_k^{\ell+1}, 
\end{equation}
where $f_k$ stands for  $f_k(\zeta)$.}

\smallskip
Notice that we do not require $H$ to be holomorphic in $\zeta$.

%%This definition is inspired by the corresponding one in
%%\cite{A7} and \cite{AW1}; a  main  difference here is that we do not
%%require $H$ to be holomorphic in $\zeta$. 

\begin{remark} We   can always find a Hefer morphism locally. 
To begin with we    can easily  find  $H_\ell^\ell$  with the stated
properties locally. 
Since $\eta$ is a complete intersection, the sheaf complex 
induced by $\nabla_\eta$ is exact except at
$\L^{0}$, whereas  a $\nabla_\eta$-closed section $\xi$ of 
$\L^{0}$ is locally exact if and only if
$\xi_{0,0}$ vanishes on the diagonal (see, e.g., \cite{A1}, Proposition~4.1). 
This latter condition is fulfilled by the right hand side of
$\nabla_\eta H^\ell_{\ell+1}=-f_{\ell+1}(z)H^{\ell+1}_{\ell+1}
+H^\ell_{\ell}f_{\ell+1}$
so there is locally a solution $H^\ell_{\ell+1}$.  One can  now  proceed by induction.
\end{remark}

Assume that $H$ is a Hefer morphism for $E_\bullet$ and let 
$U$ and $R$ be the associated currents. We can then form the
currents  $H^1U=\sum_j H^1_jU_j$ and $H^0R=\sum_j H^0_j R_j$.
To be precise with the signs one has to introduce a superbundle
structure on $E=\oplus E_k$;   then for instance $f$ is mapping of 
even order since it maps $E_k\to E_{k-1}$ (and therefore
$f$ anti-commutes with odd order forms)
whereas, e.g.,
$H$ is even since $H^\ell_k$  is a form of degree
$k-\ell$ (mod $2$)  that takes values in $\Hom(E_\ell,E_k)$,
giving another facor $k-\ell$ (mod $2$). See Section~5 in \cite{A7}
for details.

\begin{prop}\label{burdus}  
Assume that  $H$ is a Hefer morphism for the complex $E_\bullet$.
Then the current 
\begin{equation}\label{gformel}
f_1(z)H^1U+H^0R,
\end{equation}
is a weight with respect to $E_0$
and $z$ outside $Z$. 
If $\psi$ is a holomorphic section of $E_0 \otimes \mathcal{O}(-n)$,
then  we have  the representation
\begin{equation}\label{urtva}
\psi(z)=f_1(z)\int_{\Pn_\zeta}(H^1U)_{n,n}\psi+
\int_{\Pn_\zeta}(H^0R)_{n,n}\psi, \quad z\in\P^n.
\end{equation}
\end{prop}

If $R\psi=0$ we thus  have the  explicit holomorphic solution
$$
q(z)=\int_{\Pn_\zeta}(H^1U)_{n,n}\psi
$$
to $f_1 q=\psi$.

%%Here is our basic  result.

\begin{remark}\label{mainformel2}
In \cite{AW1} occur  more general currents
$U^\ell_k$ and $R^\ell_k$ taking values in $\Hom(E_k,E_\ell)$. 
With the same proof as below we get the more general formula
$$
\psi(z)=f_{\ell+1}(z)\int_{\Pn_\zeta}(H^{\ell+1}U)_{n,n}\psi+
\int_{\Pn_\zeta}(H^\ell R)_{n,n}\psi+
\int_{\Pn_\zeta}(H^\ell U)_{n,n}  f_\ell\psi
$$
for holomorphic sections of $E_\ell\otimes\mathcal{O}(-n+\ell)$.
If  $f_\ell\psi=0$ and $R\psi=0$ 
we thus get  an explicit  holomorphic solution to $f_{\ell+1}q=\psi$. 
\end{remark}

\begin{proof}[Proof of Proposition~\ref{burdus}]
The first part of the proposition  follows in the same way as the corresponding
statement (5.4) in \cite{A7}. However, we will provide  an argument 
for a more general statement, which also implies the
more general formula in Remark~\ref{mainformel2}.

To this end let $R$ and $U$ denote the ``full'' currents mentioned in the
remark. Then
\begin{equation}\label{skata}
\nabla_f\circ U+U\circ \nabla_f=I_E-R,
\end{equation}
see \cite{AW1} Section~2.  Let
\begin{equation}\label{gdef}
\tilde g=H-\nabla_\eta(HU)
\end{equation}
and let $g$  be the components that take values in
 $\oplus_\ell\Hom(E_{\zeta,\ell}, E_{z,\ell})$. 
Since $(H^\ell_\ell)_{0,0}$ is the identity on $\Delta$ it follows that
$g$ is   indeed a weight with respect
to $E$. 
Recalling that  $H$ has even order and $f$ is odd,
and using \eqref{skata} we have
\begin{multline*}
\tilde g=H- (-f(z)H+Hf)U -H\nabla_\eta U=\\
H+f(z)HU -HfU+H\dbar U= \\
H+f(z)HU -H(fU+Uf-\dbar U-Uf)=\\
H+f(z)HU-H(I_{E_\zeta}-R)+HUf.
\end{multline*}
Now  \eqref{gformel} is the component of the last term that takes values in
$\Hom(E_{\zeta,0},E_{z,0})$, and hence \eqref{gformel} is a weight.

\smallskip
The  division formula \eqref{urtva} now follows from
Proposition~\ref{hatsuyuki} for $z$ outside $Z$,
and hence in general since both sides of \eqref{urtva}
are holomorphic. One gets the formula  in Remark~\ref{mainformel2} from the component
of $\tilde g$ that takes values in $\Hom(E_{\zeta,\ell}, E_{z,\ell})$.
\end{proof}

Assume now that $E_\bullet$ is a complex with $E_k$ of the form \eqref{dirsum}
and choose $\kappa$ such that $\kappa\ge d^i_k$ for all $i,k$. We can then
construct a Hefer morphism  for the complex $E_{\bullet}\otimes\Ok(\kappa)$.
Notice that the currents $U$ and $R$ that are associated to $E_\bullet$ are
also the associated currents  to $E_{\bullet}\otimes\Ok(\kappa)$. 
We thus  obtain a division formula for sections $\psi$
of $E_0\otimes\Ok(\kappa-n)$.

Let $E'_\bullet$ denote the complex of trivial bundles over
$\C^{n+1}$ that we get from $E_\bullet$, and let $F_k$ denote the
corresponding mappings (which are formally 
 the original matrices $f_k$).
Let  $\delta_{w-z}$ denote interior multiplication with
$$
2\pi i\sum_0^n (w_j-z_j)\frac{\partial}{\partial w_j}
$$
in $\C^{n+1}_w\times\C^{n+1}_z$.

\begin{prop} \label{kowalski}
There exist $(k-\ell,0)$-form-valued 
mappings 
$$
h_k^\ell=\sum_{ij}(h_k^\ell)_{ij} e_{\ell i} \otimes e_{k j}^\ast :
\C^{n+1}_w\times\C^{n+1}_z\to\Hom(E'_k , E'_\ell)),
$$
such that $h_k^\ell = 0$ for $k<\ell$, $h_\ell^\ell = I_{E'_\ell}$, and 
\be \label{fraser}
\delta_{z-w} h_k^\ell =h_{k-1}^\ell F_k(w) - F_{\ell+1}(z) h_k^{\ell+1},
\ee
and the coefficients in the form 
$(h_k^\ell)_{ij}$ are homogeneous polynomials of degree
$d_{k}^j-d_{\ell}^i-(k-\ell)$. 
\end{prop}

For a proof, see \cite{EG2}. In Section~4 in \cite{A7} there is an explicit
formula that provides $h^\ell_k$. One has to verify, though,
that they get  the desired degrees and homogeneities.
%%
%%%
%%

\smallskip
Notice that 
\[
\alpha = \alpha_{0,0} + \alpha_{1,1} = \frac{z \cdot \bar \z}{|\z|^2} - \dbar
\frac{\bar \z \cdot d\z}{2 \pi i |\z|^2}
\]
is a well-defined smooth form in $\L^0(\Hom(\Ok_\zeta(1),\Ok_z(1)))$,
%%\P^n_\zeta\times\P^n_z$ 
such that 
\begin{equation}\label{alfasluten}
\nabla_\eta\alpha=0,
\end{equation}
and $\alpha_{0,0}$ is equal to $I_{\mathcal{O}(1)}$ on the diagonal.
Thus $\alpha$ is weight with respect to $\Ok(1)$.
Furthermore, 
\[
\gamma_j=d\zeta_j-\frac{\bar\zeta\cdot d\zeta}{|\zeta|^2}\zeta_j
\]
is  a projective form and 
\begin{equation}\label{olle2}
\nabla_\eta \gamma_j=2\pi i(z_j-\alpha \zeta_j). 
\end{equation}
%%%
If $h(w,z)$ is a homogeneous form in $\C^{n+1}_w\times\C^{n+1}_z$ with differentials $dw$ 
and polynomial  coefficients, we let $\tau^*h$ be the projective 
form obtained by
replacing $w$ by $\alpha\zeta$ and $dw_j$ by $\gamma_j$.
We then have 
\begin{equation}\label{tau}
\nabla_\eta \tau^*h=\tau^*(\delta_{w-z}h),
\end{equation}
in light of (\ref{olle2}) and \eqref{alfasluten}.

\begin{prop}
Assuming that $\kappa\ge d_k^j$ for all $k$ and $j$ we define 
\[
H^\ell_k=\sum_{ij}(\tau^\ast h_k^\ell)_{ij}\w \alpha^{\kappa-d_k^j}
e_{\ell,i}\otimes e^*_{k,j}.
\]
Then  $(H_k^\ell)$ is a  Hefer morphism for the complex $E_\bullet\otimes\mathcal{O}(\kappa)$. 
\end{prop}

\begin{proof}
First notice that $\alpha$ is in $\L^0(\mathcal{O}_z(1)\otimes\mathcal{O}_{\z}(-1))$ and hence
$\alpha \zeta_\nu$ and $z_\nu$ are in 
$\L^0(\mathcal{O}_z(1)\otimes\mathcal{O}_{\z}(0))$.  It follows that $\tau^* (h_k^\ell)_{ij}$
is a section of $\L^{-k+\ell}(\mathcal{O}_z(d_k^j-d_\ell^i)\otimes\mathcal{O}_{\z}(0))$ and hence
$\alpha^{\kappa-d_k^j}\w (\tau^\ast h_k^\ell)_{ij}$ is a section of 
$\L^{-k+\ell}(\mathcal{O}_z(\kappa-d_\ell^i)\otimes\mathcal{O}_{\z}(-\kappa+ d_k^j))$.
This means that  $H^\ell_k$ is  indeed a section
of $\L^{-k+\ell}(\Hom(E_{\zeta,k}\otimes\mathcal{O}_\zeta(\kappa),E_{z,\ell}
\otimes\mathcal{O}_z(\kappa)))$. 

It is readily checked that $(H^\ell_\ell)_{0,0}$ equals $I_{E_\ell}=I_{E_\ell\otimes\Ok(\kappa)}$ 
on the diagonal. We now show that (\ref{hrel}) holds. 
Using (\ref{olle2}),  (\ref{fraser}) and \eqref{alfasluten} 
we have that
\begin{multline*}
(\nabla_\eta H^\ell_k)_{ij} 
= \nabla_\eta [\alpha^{\kappa-d_k^j}\w\tau^*(h_k^\ell)_{ij}]
=\alpha^{\kappa-d_k^j}\w\tau^*\delta_{w-z}(h_k^\ell)_{jk}= \\
\alpha^{\kappa-d_k^j}\w
\tau^*\sum_\nu(F_{\ell+1})_{i\nu}(z)(h_k^{\ell+1})_{\nu j}+
\alpha^{\kappa-d_k^j}\w \tau^*\sum_\nu
(h_{k-1}^\ell)_{i\nu}(F_k(w))_{\nu j}.
\end{multline*}
The next to last term is 
$$
-[f(z)_{\ell+1} H_k^{\ell+1}]_{ij}
$$
and since 
$$
\tau^*(F_k(w))=\alpha^{d_k^j-d_{k-1}^\nu} f_k(\zeta),
$$
the last term is 
$$
[H_{k-1}^\ell f_k]_{ij}.
$$
Thus the proposition  follows. 
\end{proof}

\section{Examples}

Let us  compute a solution formula corresponding to the Koszul complex, 
cf., Example~\ref{koszul}. Then we  first have to find a Hefer morphism.
Let $ \tilde h_j(w,z)$ be $(1,0)$-forms in $\C^{n+1}\times\C^{n+1}$ 
of polynomial degrees $d_j-1$ such that
$$
\delta_{w-z} \tilde h_j=f^j(w)-f^j(z)
$$
and let 
$
h_j=\tau^*\tilde h_j$.  %%=\tilde h(\zeta\alpha,z)$.
%%h=\sum h_j\w 
%%
We only have to care about $k\le \min(m,n+1)$ so we assume that
$$
\kappa\ge d_1+\ldots +d_{\min(m,n+1)}.
$$
Then 
$$
H^\ell_k=\sum'_{|I|=\ell}\sum'_{|J|=k-\ell}\pm
h_{J_1}\w\ldots\w h_{J_{k-\ell}}\w e_I\otimes e^*_{IJ}\w\alpha^{\kappa-(d_{J_1}+\cdots +d_{J_{k-\ell}}
+d_{I_1}+\cdots +d_{I_k})}
$$
is a Hefer morphism.  The components of most interest for us are 
$H_k^0$ and $H_k^1$. Since 
$$
H_k^0=\sum'_{|J|=k}\pm
h_{J_1}\w\ldots\w h_{J_{k}}\w  e^*_{J}\w\alpha^{\kappa-(d_{J_1}+\cdots +d_{J_{k-\ell}})},
$$
it can be more compactly written  as 
$$
H_k^0=\alpha^\kappa \w (\delta_{\hat h})_k
$$
where $\delta_{\hat h}$ denotes formal interior multiplication with
$$
\hat h=\sum \alpha^{-d_j}\w h_j\w e_j^*
$$
and $(\delta_{\hat h})_k=\delta_{\hat h}^k/k!.$
In the same way
$$
H_k^1=\alpha^{\kappa}\w N(\delta_{\hat h})_{k-1},
$$
where
$$
N=\sum_j \alpha^{-d_j}e_j\otimes e_j^*.
$$

\smallskip

Assume now that $\psi$ is a section of $\mathcal{O}(\rho)$ where   $\rho=\kappa-n$. 
We then have the decomposition \eqref{urtva}.
If  in addition $R\phi=0$ we thus  have  that
$$
\psi(z)=\sum f^j(z)\cdot q_j(z)=\delta_{f(z)} q(z)
$$
where, cf., Example~\ref{koszul},
\begin{equation}\label{hurra}
q(z)=\sum_{k=1}^{\min(m,n+1)}
\int_{\Pn_\zeta} \alpha^\kappa\w N (\delta_{\hat h})_{k-1} (\sigma\w (\dbar\sigma)^{k-1}) \psi.
\end{equation}
With the assumptions in Theorem~\ref{hickelthm2} or \ref{newthm}, 
the proofs in Section~\ref{proofsec} show that
$\psi=z_0^{\rho-\deg\Phi}\phi$ annihilates the residue $R$,
for appropriate choice of $\rho$,  and hence
\eqref{hurra} is an explicit solution to the division problem
(after dehomogenization).

\begin{ex} As in Example~\ref{ex2} let us assume that $d_j=d$.  
If  $h=\sum h_j\w e_j^*$, then 
$$
H^1_k=\alpha^{\kappa -dk}(\delta_h)_{k-1}.
$$
Then our solution $q$ in  \eqref{hurra}   takes values in $\Ok(\rho-d)=\Ok(\kappa-n-d)$,
and it can be written as 
\begin{equation}\label{hurra2}
q=\sum_{k=1}^{\min(m,n+1)}
\int_{\Pn} \alpha^{\rho+n-dk}\w (\delta_{h})_{k-1} \frac{\bar f\cdot e\w(d\bar f\cdot e)^{k-1}}
{|f|^{2k}}\psi.
\end{equation}
\end{ex}

If $f$ has no zeros at all, then 
\eqref{hurra} provides a completely explicit representation
of Macaulay's theorem  and the integrand is smooth. 
If $m\le n$ and  $|\psi|\le |f|^{\min(m,n)}$, then 
$U\psi$ is  integrable  so 
\eqref{hurra} (and in particular \eqref{hurra2}) is  a convergent integral. In general,
however, $U\psi$ may be  a distribution of  higher order than zero,
and then \eqref{hurra} must be regarded as a principal value.
For instance one can multiply by $|f|^{2\lambda}$ and put
$\lambda=0$, cf., Example~\ref{koszul}. One can just as well
define it as a classical principal value, see, e.g., \cite{A2}, 
$$
\lim_{\epsilon\to 0}
\sum_{k=1}^{\min(m,n+1)}
\int_{\Pn} \chi(|f|/\epsilon)
\alpha^\kappa\w N (\delta_h)_{k-1} (\sigma\w (\dbar\sigma)^{k-1}) \psi,
$$
where $\chi(t)$ is (a smooth approximand of) the characteristic function
for the interval $[1,\infty)$.

\begin{remark}
If $f^j$ are rational we can choose rational  Hefer polynomials $\tilde h_j$,
and  then $h_j$ are rational expressions in $\alpha\zeta$, $z$, and $\gamma_i$.
It is possible that the resulting solution then actually is rational if
$\psi$ is rational but we have no argument.
\end{remark}

\smallskip

Also in the case of the Buchsbaum-Rim complex
one can find quite simple
expressions for a Hefer morphism for the corresponding
homogeneous complex,  see Section~6 in \cite{A7}. One then
obtain the projective Hefer morphism following  Section~\ref{formelsec} above.
Therefore we get an explicit  
representation   for the solutions in Theorem~\ref{hickelthm2matris}
as well; however we omit the details.

\def\listing#1#2#3{{\sc #1}:\ {\it #2},\ #3.}

\end{document}